\documentclass[12pt,  reqno]{amsart}

\usepackage[T1]{fontenc}
\usepackage[utf8]{inputenc}

\usepackage{latexsym}
\usepackage{amsmath}
\usepackage{amsfonts}
\usepackage{a4wide}
\usepackage{amssymb}
\usepackage{amsthm}
\usepackage[colorlinks, citecolor=blue, linkcolor=red]{hyperref}

\theoremstyle{plain}

\newtheorem{teo}{Theorem}[section]
\newtheorem{lemma}[teo]{Lemma}

\newtheorem{cor}[teo]{Corollary}
\newtheorem{ackn}{Acknowledgments\!}

\theoremstyle{definition}

\theoremstyle{remark}

\newtheorem{rem}[teo]{Remark}

\numberwithin{equation}{section}

\def\kn{\mathbin{\bigcirc\mkern-15mu\wedge}}

\def\SS{{{\mathbb S}}}

\def\RR{{\mathbb R}}

\def\gt{\widetilde{g}}

\def\eps{\varepsilon}

\def\gt{\widetilde{g}}
\def\gb{\overline{g}}
\def\vphi{\varphi}

\newcommand{\set}[1]{{\left\{#1\right\}}}               
\newcommand{\pa}[1]{{\left(#1\right)}}                  
\newcommand{\sq}[1]{{\left[#1\right]}}                  
\newcommand{\abs}[1]{{\left|#1\right|}}                 

\title[Metrics of constant negative scalar-Weyl curvature]{Metrics of constant negative scalar-Weyl curvature}

\author[Giovanni Catino]{Giovanni Catino}
\address[Giovanni Catino]{Dipartimento di Matematica, Politecnico di Milano, Piazza Leonardo da Vinci 32, 20133 Milano, Italy}
\email[]{giovanni.catino@polimi.it}

\begin{document}

\begin{abstract} Extending Aubin's construction of metrics with constant negative scalar curvature, we prove that every $n$-dimensional closed manifold admits a Riemannian metric with constant negative scalar-Weyl curvature, that is $R+t|W|, t\in\mathbb{R}$. In particular, there are no topological obstructions for metrics with $\varepsilon$-pinched Weyl curvature and negative scalar curvature.
\end{abstract}

\maketitle

%
%
%
%

\section{Introduction}

A natural problem in Riemannian geometry is to understand the relation between curvature and topology of the underlying manifold. Given a smooth $n$-dimensional manifold $M$, $n\geq 3$, the curvature tensor of  a Riemannian metric $g$ on $M$ can be decomposed in its Weyl, Ricci and scalar curvature part, that is
$$
Riem_g = W_g+\frac{1}{n-2} Ric_g \kn g - \frac{R_g}{2(n-1)(n-2)}g\kn g,
$$
where $\kn$ is the Kulkarni-Nomizu product. It is common knowledge that {\em weak positive} curvature conditions, such as positive scalar curvature $R_g$ \cite{schyau, grolaw}, or {\em strong negative}  ones, such as negative sectional curvature, are in general obstructed. On the other hand, Aubin in \cite{aub1, aub2} showed that, on every smooth $n$-dimensional closed (compact with empty boundary) manifold, there exists a smooth Riemannian metric with  constant negative scalar curvature, $R_g\equiv -1$. This result was extended to the complete, non-compact, case by Bland and Kalka in \cite{blakal}. In particular, there are no topological obstructions for negative scalar curvature metrics. Actually, a much stronger result is known: Lohkamp in \cite{loh} proved that every smooth  $n$-dimensional complete manifold admits a complete smooth Riemannian metric with (strictly) negative Ricci curvature, $Ric_g<0$ (the three dimensional case was considered in \cite{gaoyau, bro}). 

By virtue of the Riemann components, in dimension $n\geq 4$, it is natural to ask if there are unobstructed curvature conditions which involves the Weyl curvature. To the best of our knowledge, the first result in this direction was proved by Aubin \cite{aub2}, who constructed a metric with non-zero Weyl curvature on every closed $n$-dimensional manifold. As a consequence, in \cite{cmmp} the authors proved the existence of a canonical metric (weak harmonic Weyl)  whose Weyl tensor satisfies a second order Euler-Lagrange PDE, on every given closed four-manifold. 

In \cite{gur1}, Gursky studied a variant of the Yamabe problem related to a modified scalar curvature given by 
$$
R_g+t|W_g|_g,\quad t\in\RR,
$$
where $|W_g|_g$ denotes the norm of the Weyl curvature of $g$. We will refer to this quantity as the {\em scalar-Weyl curvature} (see Section \ref{s-swc}). Constant scalar-Weyl curvature metrics naturally arise as critical points in the conformal class of the modified Einstein-Hilbert functional
$$
g\longmapsto \text{Vol}_g(M)^{-\frac{n-2}{2}}\int_M \pa{R_g+t|W_g|_g}\,dV_g.
$$
It is clear that positive scalar-Weyl curvature metrics are obstructed, at least for $t\leq 0$, and naturally we may ask what we can say concerning the negative regime. In this paper we prove the following existence result:

\begin{teo}\label{t-main} On every smooth $n$-dimensional closed manifold $M$, for every $t\in\RR$, there exists a smooth Riemannian metric $g=g_{t}$ with 
$$
R_g+t|W_g|_g\equiv -1\quad\text{on } M.
$$
In particular, there are no topological obstructions for negative scalar-Weyl curvature metrics.
\end{teo}

\begin{rem} In dimension four, Theorem \ref{t-main} was proved also by Seshadri in \cite{ses}. We observe that his proof cannot be trivially generalized to higher dimension, since it is based on the existence of a hyperbolic metric on a knot complement of $\SS^3$.
\end{rem}

It is well known that there are obstructions for the existence of metrics with zero Weyl curvature. On the other hand, choosing $t=1/\sqrt{\eps}$, $\eps>0$,  in Theorem \ref{t-main} we obtain the following existence result for metrics with $\eps$-pinched Weyl curvature and negative scalar curvature:

\begin{cor} On every smooth $n$-dimensional closed manifold, for every $\varepsilon>0$, there exists a smooth Riemannian metric $g=g_{\varepsilon}$ with 
$$
R_g<0 \quad\text{and}\quad |W_g|_g^2 <\varepsilon R_g^2\quad\text{on } M.
$$
\end{cor}
The interesting notion of {\em isotropic curvature} was introduced  by Micallef and Moore in \cite{micmoo}:  $(M,g)$ has positive (or negative) isotropic curvature if and only if the curvature tensor of $g$ satisfies
$$
R_{1313}+R_{1414}+R_{2323}+R_{2424}-2R_{1234}>0\quad(\text{or } <0)
$$
for all orthonormal $4$-frames $\{e_1,e_2,e_3,e_4\}$. In \cite{micmoo}, using minimal surfaces, the author proved that any closed simply connected manifold with positive isotropic curvature is homeomorphic to the sphere $\SS^n$.
As already observed in \cite{ses}, in dimension four, metrics with negative scalar-Weyl curvature for $t\geq 6$ have negative isotropic curvature. In particular, Theorem \ref{t-main} implies the following:

\begin{cor}  On every smooth four-dimensional orientable closed manifold there exists a smooth Riemannian metric  with negative isotropic curvature.
\end{cor}

\

\section{The scalar-Weyl curvature}\label{s-swc}

In this section we briefly recall the variational and conformal aspects of the scalar-Weyl curvature, firstly studied by Gursky in \cite{gur1}. Let $(M, g)$ be a $n$-dimensional closed (compact with empty boundary) Riemannian manifold. First we recall that the conformal Laplacian is the operator
$$
\mathcal{L}_g:=-\frac{4(n-1)}{n-2}\Delta_g+R_g,
$$
which has the following well known conformal covariance property: if $\gt=u^{4/(n-2)} g$, then 
$$
\mathcal{L}_{\gt}\phi = u^{-\frac{n+2}{n-2}}\mathcal{L}_g(\phi u),\quad \forall \phi\in C^2(M).
$$
Moreover, the scalar curvature of the conformally related metric $\gt$ is given by
$$
R_{\gt}=u^{-\frac{n+2}{n-2}} \mathcal{L}_g u.
$$
Therefore, the operator $\mathcal{L}$ plays a prominent role in the resolution of the Yamabe variational problem. Given $t\in\RR$, we define the scalar-Weyl curvature
\begin{equation}\label{scalar-weyl}
F_g:= R_g + t |W_g|_g
\end{equation}
and the associated modified conformal Laplacian
$$
\mathcal{L}^t_g:=-\frac{4(n-1)}{n-2}\Delta_g+F_g,
$$
where $|W_g|_g$ denotes the norm of the Weyl curvature of $g$. The key observation in \cite{gur1} is that the couples $(F_g, \mathcal{L}^t_g)$ and $(R_g, \mathcal{L}_g)$ share the same conformal properties. In fact, if $\gt=u^{4/(n-2)} g$, then
\begin{equation}\label{e-confF}
\mathcal{L}^t_{\gt}\phi = u^{-\frac{n+2}{n-2}}\mathcal{L}^t_g(\phi u),\quad \forall \phi\in C^2(M),\quad\text{and}\quad F_{\gt}=u^{-\frac{n+2}{n-2}} \mathcal{L}^t_g u.
\end{equation}
In particular, a spectral argument shows the following \cite[Proposition 3.2]{gur1}:
\begin{lemma}\label{l-g1} Let $(M,g)$ be a $n$-dimensional closed Riemannian manifold. Then, there exists  a $C^{2,\alpha}$ metric $\gt\in[g]$ with either $F_{\gt}>0$, $F_{\gt}<0$, or $F_{\gt}\equiv 0$. Moreover, these three possibilities are mutually exclusive. 
\end{lemma}
In analogy with the Yamabe problem, Gursky defined the functional
$$
\widehat{Y}(u):=\frac{\int_M u\, \mathcal{L}^t_g u\,dV_g}{\left(\int_M u^{2n/(n-2)}\,dV_g\right)^{(n-2)/2}}
$$
and the conformal invariant
$$
\widehat{Y}(M,[g]):=\inf_{u\in H^1(M)} \widehat{Y}(u).
$$
Using \eqref{e-confF}, it is easy to see that the functional $u\mapsto\widehat{Y}(u)$ is equivalent to the modified Einstein-Hilbert functional
$$
\gt=u^{4/(n-2)}g \longmapsto \frac{\int_M F_{\gt}\,dV_{\gt}}{\text{Vol}_{\gt}(M)^{(n-2)/2}}.
$$
Following a classical subcritical regularization argument, Gursky showed that, if $\widehat{Y}(M,[g])\leq 0$, then the variational problem of finding a conformal metric $\gt\in[g]$ with constant scalar-Weyl curvature $F$ can be solved. The proof (in dimension four) can be found in \cite[Proposition 3.5]{gur1} and it can be trivially generalized to dimension $n\geq 4$. In particular, we have the following sufficient condition to the existence of constant negative scalar-Weyl curvature:
\begin{lemma}\label{l-g2} Let $(M,g)$ be a $n$-dimensional closed Riemannian manifold. If there exists a metric $g'\in[g]$ such that
$$
\int_M F_{g'}\,dV_{g'} <0,
$$
then, there exists a (unique) $C^{2,\alpha}$ metric $\gt\in[g]$ such that $F_{\gt}\equiv -1$.
\end{lemma}

To conclude this section, we observe that the full modified Yamabe problem related to the scalar-Weyl curvature and more generally modified scalar curvatures was treated in \cite{ito}. Moreover, these techniques introduced by Gursky, have been used in various contexts, especially in the four-dimensional case. For instance we want to highlight \cite{gurleb1, gurleb2, leb, ses}.

\

\section{Aubin's metric deformation: two integral inequalities} 

In this section we first recall the variational formulas for some geometric quantities under the deformation of the metric of the type
$$
g'=g+df\otimes df,\quad f\in C^{\infty}(M).
$$ 
In \cite{aub1, aub2} Aubin, with a clever coupling of this deformation with a conformal one, proved local and global existence results of metrics satisfying special curvature conditions. The proof of the first three formulas can be found in \cite{aub2}. The variation of the Weyl tensor can be found in \cite[Chapter 2]{catmasbook}.
\begin{lemma}\label{l-a0} Let $(M,g)$ be a $n$-dimensional Riemannian manifold and consider the variation of the metric $g$, in a given local coordinate system, defined by
$$
g'_{ij}:=g_{ij}+f_i f_j,\quad f\in C^{\infty}(M).
$$
Then we have
\begin{align*}
dV_{g'} &= w^{1/2} dV_g,\\
(g')^{ij} &= g^{ij} - \frac{f^if^j}{w},\\
R' &= R -\frac{2}{w}R_{ij}f^if^j + \frac{1}{w}\sq{\pa{\Delta f}^2-f_{it}f^{it}} -\frac{2}{w^2}\sq{\pa{\Delta f}f^if^jf_{ij}-f^if_{ij}f^{jp}f_p},\\
W'_{ijkt} &= W_{ijkt} + E(f)_{ijkt},
\end{align*}
with $w:=1+|\nabla f|^2$ and 
\begin{align*}
&E(f)_{ijkt}:=\frac{1}{w}\pa{f_{ik}f_{jt}-f_{it}f_{jk}} +\frac{1}{n-2}\pa{R_{ik}f_jf_t-R_{it}f_jf_k+R_{jt}f_if_k-R_{jk}f_if_t} \\ \nonumber &+\frac{R}{(n-1)(n-2)}\pa{g_{ik}f_jf_t-g_{it}f_jf_k+g_{jt}f_if_k-g_{jk}f_if_t} \\ \nonumber &+\frac{f^pf^q}{w(n-2)}\sq{R_{ipkq}\pa{g_{jt}+f_jf_t}-R_{iptq}\pa{g_{jk}+f_jf_k}+R_{jptq}\pa{g_{ik}+f_if_k}-R_{jpkq}\pa{g_{it}+f_if_t}} \\ \nonumber &-\frac{2R_{pq}f^pf^q}{w(n-1)(n-2)}\sq{g_{ik}g_{jt}-g_{it}g_{jk}+g_{ik}f_jf_t-g_{it}f_jf_k+g_{jt}f_if_k-g_{jk}f_if_t} \\ \nonumber &-\frac{1}{w(n-2)}\set{\sq{(\Delta f) f_{ik}-f_{ip}f^p_k}\pa{g_{jt}+f_jf_t}-\sq{(\Delta f) f_{it}-f_{ip}f^p_t}\pa{g_{jk}+f_jf_k}} \\ \nonumber &-\frac{1}{w(n-2)}\set{\sq{(\Delta f) f_{jt}-f_{jp}f^p_t}\pa{g_{ik}+f_if_k}-\sq{(\Delta f) f_{jk}-f_{jp}f^p_k}\pa{g_{it}+f_if_t}} \\ \nonumber &+\frac{1}{w(n-1)(n-2)}\sq{\pa{\Delta f}^2-\abs{\nabla^2 f}^2}\pa{g_{ik}g_{jt}-g_{it}g_{jk}+g_{ik}f_jf_t-g_{it}f_jf_k+g_{jt}f_if_k-g_{jk}f_if_t} \\ \nonumber &+\frac{f^pf^q}{w^2(n-2)}\sq{\pa{f_{ik}f_{pq}-f_{ip}f_{kq}}\pa{g_{jt}+f_jf_t}-\pa{f_{it}f_{pq}-f_{ip}f_{tq}}\pa{g_{jk}+f_jf_k}} \\ \nonumber &+\frac{f^pf^q}{w^2(n-2)}\sq{\pa{f_{jt}f_{pq}-f_{jp}f_{tq}}\pa{g_{ik}+f_if_k}-\pa{f_{jk}f_{pq}-f_{jp}f_{kq}}\pa{g_{it}+f_if_t}}  \\ \nonumber &-\frac{2}{w^2(n-1)(n-2)}\sq{(\Delta f)f^pf^qf_{pq}-f^pf_{pq}f^{qr}f_r}\pa{g_{ik}g_{jt}-g_{it}g_{jk}} \\ \nonumber &-\frac{2}{w^2(n-1)(n-2)}\sq{(\Delta f)f^pf^qf_{pq}-f^pf_{pq}f^{qr}f_r}\pa{g_{ik}f_jf_t-g_{it}f_jf_k+g_{jt}f_if_k-g_{jk}f_if_t}.
\end{align*}
Moreover,
$$
R' = R - \frac{R_{ij}f^if^j}{w} + \nabla^i \left( \frac{\Delta f f_i-f_{ij}f^j}{w}\right)
$$
and thus
$$
\int_M R'\,dV_g = \int_M R\,dV_g - \int_M  \frac{R_{ij}f^if^j}{1+|\nabla f|^2}\,dV_g.
$$
\end{lemma}

We will denote by $[g]$ the conformal class of the metric $g$. Using a conformal deformation, we can show the following first integral sufficient condition for the existence of a constant negative scalar-Weyl curvature:
\begin{lemma}\label{l-a1} Let $M$ be a $n$-dimensional closed manifold. If there exists a positive smooth function $u\in C^{\infty}(M)$ such that for a Riemannian metric $g$ on $M$ it holds
$$
\int_M F_g\,u^2\,dV_g + \frac{4(n-1)}{n-2}\int_M |\nabla u|^2\,dV_g <0,
$$
then there exists a (unique) $C^{2,\alpha}$ metric $\gt\in[g]$ such that $F_{\gt}\equiv -1$.
\end{lemma}
\begin{proof} We consider the conformal metric $g'_{ij}=u^{4/(n-2)} g$. By \eqref{e-confF} we have
$$
F_{g'}=R_{g'}+t|W_{g'}|_{g'}= u^{-4/(n-2)}\left(R_g+t|W_g|_g-\frac{4(n-1)}{n-2}\frac{\Delta u}{u}\right).
$$
Therefore, since $dV_{g'}=u^{2n/(n-2)}dV_g$, using the assumption we obtain
$$
\int_M F_{g'}\,dV_{g'} = \int_M F_g\,u^2\,dV_g + \frac{4(n-1)}{n-2}\int_M |\nabla u|^2\,dV_g <0.
$$
The conclusion follows now by Lemma \ref{l-g2}.
\end{proof}
Using Aubin's deformations, we prove the following second integral sufficient condition for the existence of a constant negative scalar-Weyl curvature:
\begin{lemma}\label{l-a2} Let $M$ be a $n$-dimensional closed manifold. Suppose that there exists a smooth function $\vphi\in C^{\infty}(M)$ such that for a Riemannian metric $g$ on $M$ and some $t>0$ it holds
\begin{align*}
\int_M &\left(R_g+t|W_g|_{\vphi}\right)\,dV_g + t\int_M |E_g(\vphi)|_{\vphi}\,dV_g \\
&-\int_M \frac{R_{ij}\vphi^i\vphi^j}{1+|\nabla \vphi|^2}\,dV_g +\frac{n-1}{n-2}\int_M \left[\frac{\vphi_{ip}\vphi^p\vphi_{iq}\vphi^q}{(1+|\nabla\vphi|^2)^2} -\frac{|\vphi_{ij}\vphi^i\vphi^j|^2}{(1+|\nabla\vphi|^2)^3} \right]\,dV_g<0,
\end{align*}
where  $|\cdot|_{\vphi}$ denotes the norm with respect of $g+d\vphi\otimes d\vphi$ and $E_g(\vphi)$ is defined as in Lemma \ref{l-a0}. Then, there exists a (unique) $C^{2,\alpha}$ metric $\gt\in[g+d\vphi\otimes d\vphi]$ such that $F_{\gt}\equiv -1$.
\end{lemma}
\begin{proof} Let $\vphi\in C^{\infty}(M)$. Applying Lemma \ref{l-a1} to the metric $g'=g+d\vphi\otimes d\vphi$ with 
$$
u:= \left(1+|\nabla\vphi|^2\right)^{-1/4},
$$
we know that there exists a conformal metric $g''\in[g']$ with $F_{g''}\equiv -1$, if
\begin{align*}
\int_M \frac{F_{g'}}{\left(1+|\nabla\vphi|^2\right)^{1/2}}\,dV_{g'}+\frac{4(n-1)}{n-2}\int_M \left|\nabla \left(1+|\nabla\vphi|^2\right)^{-1/4}\right|_{g'}^2\,dV_{g'}<0.
\end{align*}
From Lemma \ref{l-a0} we obtain the equivalent inequality
\begin{align*}
\int_M &F_{g'}\,dV_{g}+\frac{4(n-1)}{n-2}\int_M \partial_i\left(1+|\nabla\vphi|^2\right)^{-1/4}\partial_j\left(1+|\nabla\vphi|^2\right)^{-1/4}\left(g^{ij}-\frac{\vphi_i\vphi_j}{1+|\nabla\vphi|^2}\right)\,dV_{g'}\\
&= \int_M F_{g'}\,dV_{g}+\frac{n-1}{n-2}\int_M \left[\frac{\vphi_{ip}\vphi^p\vphi_{iq}\vphi^q}{(1+|\nabla\vphi|^2)^2} -\frac{|\vphi_{ij}\vphi^i\vphi^j|^2}{(1+|\nabla\vphi|^2)^3} \right]\,dV_g<0.
\end{align*}
Using again Lemma \ref{l-a0}, we get
$$
\int_M F_{g'}\,dV_g= \int_M \left(R_{g'}+t|W_{g'}|_{g'}\right)\,dV_g =  \int_M \left(R_{g}+t|W_{g'}|_{g'}\right)\,dV_g - \int_M \frac{R_{ij}\vphi^i\vphi^j}{1+|\nabla \vphi|^2}\,dV_g.
$$
Using that
$$
|W_{g'}|_{g'} \leq |W_g|_{g'} + |E_g (\vphi)|_{g'}
$$
where $E_g(\vphi)$ is defined as in Lemma \ref{l-a0}, we conclude the proof of this lemma.
\end{proof}

\

\section{Proof of Theorem \ref{t-main}}

In this section we prove Theorem \ref{t-main}. The strategy of the proof takes strong inspiration from the works of Aubin in \cite{aub1,aub2}.

\medskip

\subsection*{Step 1.} From \cite{aub1, aub2} we know that, on a closed $n$-dimensional manifold, there exists a Riemannian metric $g'$ with constant scalar curvature $-1$. In particular, if $t\leq 0$, $F_{g'}<0$. By Lemma  \ref{l-g2}, there exists a metric $\gt\in[g']$ such that $F_{\gt}\equiv -1$. Therefore, from now on we focus on the case
$$
t>0.
$$

First of all, we can choose a Riemannian metric $g$ with 
$$F_g=R_g+t|W_g|_g\geq 0\quad\text{on } M,$$ 
otherwise Theorem \ref{t-main} would immediately follow from Lemma \ref{l-g1} and Lemma \ref{l-g2}. Consider a positive smooth function $\psi\in C^\infty (M)$ and a positive constant $k>0$, and define
$$
g' := \psi g,\quad g'' := g' +d(k\psi)\otimes d(k\psi).
$$
If we fix $t>0$ and apply Lemma \ref{l-a2} to the metric $g'$ with $\vphi=k\psi$, we obtain that if
\begin{align*}
\Phi_M:=\int_M &\left(R_{g'}+t|W_{g'}|_{k\psi}\right)\,dV_{g'} + t\int_M |E_{g'}(k\psi)|_{k\psi}\,dV_{g'} -\int_M \frac{R'_{ij}\nabla_{g'}^i\psi\nabla_{g'}^j\psi}{1/k^2+|\nabla_{g'} \psi|_{g'}^2}\,dV_{g'}\\
& +\frac{n-1}{n-2}\int_M \left[\frac{\nabla^{g'}_{ip}\psi\nabla_{g'}^p\psi\nabla^{g'}_{iq}\psi\nabla_{g'}^q\psi}{(1/k^2+|\nabla_{g'}\psi|_{g'}^2)^2} -\frac{|\nabla^{g'}_{ij}\psi\nabla_{g'}^i\psi\nabla_{g'}^j\psi|^2}{(1/k^2+|\nabla_{g'}\psi|_{g'}^2)^3} \right]\,dV_{g'}<0,
\end{align*}
then there exists a (unique) $C^{2,\alpha}$ metric $\gt\in[g'']$ such that $F_{\gt}\equiv-1$. Therefore, to prove Theorem \ref{t-main}, it is sufficient to show that $\Phi_M<0$ for some positive smooth function $\psi$ and positive constant $k$ (concerning the regularity of the metric, see the end of the proof). Let
$$
f:=\psi^{(n-2)/2}.
$$
With respect to the metric $g$, by standard formulas for conformal transformations (see \cite[Chapter 5]{catmasbook}), we have
\begin{align}\label{e-confinv}\nonumber
R_{g'}&=\frac{1}{\psi}\left(R_g-\frac{2(n-1)}{n-2}\frac{\Delta f}{f}+\frac{n-1}{n-2}\frac{|\nabla f|^2}{f^2}\right),\\\nonumber
R'_{ij}&=R_{ij}-\frac{f_{ij}}{f}+\frac{n-1}{n-2}\frac{f_if_j}{f^2}-\frac{1}{n-2}\frac{\Delta f}{f}g_{ij},\\
W'_{ijkt}&=\frac{1}{\psi} W_{ijkt},\\\nonumber
dV_{g'}&=\psi^{n/2}\,dV_g=f \psi\,dV_g,\\\nonumber
\nabla^{g'}_{ij} \psi &= \psi_{ij} -\frac{1}{\psi}\left(\psi_i\psi_j-\frac12|\nabla\psi|^2 g_{ij}\right).  
\end{align}
Moreover, since 
$$
g''= g' +d(k\psi)\otimes d(k\psi) = \psi \left[ g + d(2k\sqrt{\psi})\otimes d(2k\sqrt{\psi})\right] =:\psi \gb,
$$
from the conformal invariance of the Weyl curvature and Lemma \ref{l-a0}, we obtain
$$
W'_{ijkt}+E'(k\psi)_{ijkt} = W''_{ijkt} = \frac{1}{\psi}\overline{W}_{ijkt} = \frac{1}{\psi}\left[W_{ijkt}+E(2k\sqrt{\psi})_{ijkt} \right]=W'_{ijkt}+\frac{1}{\psi}E(2k\sqrt{\psi})_{ijkt}.
$$
Therefore, the "error term" of Weyl tensor under Aubin's deformation of the metric satisfies the following {\em conformal invariance}:
\begin{equation}\label{e-confinvE}
E_{g'}(k\psi) =\frac{1}{\psi} E_g (2k\sqrt{\psi}).
\end{equation}
In particular, we have the relations
$$
|W_{g'}|_{k\psi}=|W_{g'}|_{g'+d(k\psi)\otimes d(k\psi)} = \frac{1}{\psi}|W_{g'}|_{\gb} = \frac{1}{\psi^2}|W_g|_{\gb}
$$
and
$$
|E_{g'}(k\psi)|_{k\psi}=\frac{1}{\psi}|E_{g'}(k\psi)|_{\gb} =\frac{1}{\psi^2} |E_{g}(2k\sqrt{\psi})|_{\gb}.
$$
Following the computation in \cite{aub2}, putting all together we obtain
\begin{align*}
\Phi_M:=&\int_M \left(R_{g}+\frac{t}{\psi}|W_{g}|_{\gb}-\frac{R_{ij}\psi_i\psi_j}{\psi/k^2+|\nabla\psi|^2}\right)f\,dV_{g} + t\int_M \frac{f}{\psi}|E_{g}(2k\sqrt{\psi})|_{\gb}\,dV_{g} \\
&+\int_M \frac{f_{ij}\psi^i\psi^j}{\psi/k^2+|\nabla\psi|^2}\,dV_g+\frac{n-1}{n-2}\int_M \frac{|\nabla f|^2}{f}\,dV_g-\frac{n-1}{n-2}\int_M \frac{|f_i\psi^i|^2}{f(\psi/k^2+|\nabla\psi|^2)}\,dV_g\\
&+\frac{1}{n-2}\int_M \frac{\Delta f |\nabla\psi|^2}{\psi/k^2+|\nabla\psi|^2}\,dV_g\\
&+\frac{n-1}{n-2}\int_M \left[\frac{\psi_{ip}\psi^p\psi_{iq}\psi^q}{(\psi/k^2+|\nabla\psi|^2)^2} -\frac{|\psi_{ij}\psi^i\psi^j|^2}{(\psi/k^2+|\nabla\psi|^2)^3} \right]f\,dV_{g}\\
&+\frac{1}{k^2}\frac{n-1}{n-2}\int_M \frac{\tfrac14 |\nabla\psi|^6-|\nabla\psi|^2(\psi_{ij}\psi^i\psi^j)\psi}{(\psi/k^2+|\nabla\psi|^2)^3}f\,dV_g.
\end{align*}
Moreover, since
$$
\int_M \frac{|\nabla f|^2}{f}\,dV_g-\int_M \frac{|f_i\psi^i|^2}{f(\psi/k^2+|\nabla\psi|^2)}\,dV_g=\frac{1}{k^2}\frac{n-2}{2}\int_M \frac{f_i\psi^i}{\psi/k^2+|\nabla\psi|^2}\,dV_g,
$$
$$
\int_M \frac{\Delta f |\nabla\psi|^2}{\psi/k^2+|\nabla\psi|^2}\,dV_g = -\frac{1}{k^2}\int_M \frac{\psi\Delta f}{\psi/k^2+|\nabla\psi|^2}\,dV_g,
$$
we finally get
\begin{align}\label{e-int}\nonumber
\Phi_M:=&\int_M \left(R_{g}+\frac{t}{\psi}|W_{g}|_{\gb}-\frac{R_{ij}\psi_i\psi_j}{\psi/k^2+|\nabla\psi|^2}\right)f\,dV_{g} + t\int_M \frac{f}{\psi}|E_{g}(2k\sqrt{\psi})|_{\gb}\,dV_{g} \\\nonumber
&+\int_M \frac{f_{ij}\psi^i\psi^j}{\psi/k^2+|\nabla\psi|^2}\,dV_g\\
&+\frac{1}{k^2}\frac{n-1}{2}\int_M \frac{f_i\psi^i}{\psi/k^2+|\nabla\psi|^2}\,dV_g-\frac{1}{k^2}\int_M \frac{\psi\Delta f}{\psi/k^2+|\nabla\psi|^2}\,dV_g\\\nonumber
&+\frac{n-1}{n-2}\int_M \left[\frac{\psi_{ip}\psi^p\psi_{iq}\psi^q}{(\psi/k^2+|\nabla\psi|^2)^2} -\frac{|\psi_{ij}\psi^i\psi^j|^2}{(\psi/k^2+|\nabla\psi|^2)^3} \right]f\,dV_{g}\\\nonumber
&+\frac{1}{k^2}\frac{n-1}{n-2}\int_M \frac{\tfrac14 |\nabla\psi|^6-|\nabla\psi|^2(\psi_{ij}\psi^i\psi^j)\psi}{(\psi/k^2+|\nabla\psi|^2)^3}f\,dV_g.
\end{align}

\medskip

\subsection*{Step 2.} Let $y=y(x)$ be a fixed smooth real function such that 
$$
\begin{cases}
y(-x)=y(x)\quad &\forall\, x\in\RR \\
y(x)=1\quad&\forall\, |x|\geq 1\\
y(x)\geq\delta>0\quad&\forall\, x\in\RR\\
y'(x)>0\quad&\forall\, 0<x<1\\
y'(x)\geq 1\quad&\forall\, (1/4)^{1/(n-1)}\leq x\leq (3/4)^{1/(n-1)}.
\end{cases}
$$
Let $p\in M$ and consider a local, normal, geodesic polar coordinate system around $p$: $\rho,\phi_1,\cdots,\phi_{n-1}$. We have $g_{\rho\rho}=1$, $g_{\rho i}=0$, $g_{ij}=\delta_{ij}+\rho^2 a_{ij}$, $g^{\rho\rho}=1$ (from now on, the indices $i=1,\ldots, n-1$ correspond to the coordinate $\phi_i$). The coefficient $a_{ij}$ are of order $1$. In particular, we have that the Christoffel symbols of the metric $g$ satisfy
\begin{equation}\label{e-chr}
\Gamma_{\rho\rho}^{\rho}=0,\quad \Gamma_{\rho i}^{\rho}=0,\quad \Gamma_{ij}^{\rho}=-\frac{\rho}{2}\left( a_{ij}+\rho \partial_{\rho}a_{ij}\right).
\end{equation}
Let $B_r=B_r(p)$ be the geodesic ball centered at $p$ of radius $0<r<r_0$, with $r_0$ such that $B_r\subset M$. For $p'\in B_r$, we choose
$$
f(p'):=y\left(\frac{\rho}{r}\right),\quad \rho=\text{dist}_g(p',p).
$$
In particular, from \eqref{e-chr}, we have
\begin{equation}\label{e-f1}
f_{\rho}(p')=\frac{1}{r}y'\left(\frac{\rho}{r}\right),\quad f_i(p')=0,
\end{equation}
\begin{equation}\label{e-f1}
f_{\rho\rho}(p')=\frac{1}{r^2}y''\left(\frac{\rho}{r}\right), \quad f_{\rho i}(p')=0, \quad f_{ij}(p')= \frac{\rho}{2r}\left( a_{ij}+\rho \partial_{\rho}a_{ij}\right)y'\left(\frac{\rho}{r}\right).
\end{equation}
From now on, to simplify the expressions, we will omit the dependence of the functions: it will be clear that if $f$, $f_\rho$, etc. are computed at $p'\in B_r$, then $y, y', y''$ will be computed at $\rho/r$ with $\rho=\text{dist}_g(p',p)$. Moreover, we will denote by $C=C(n,\delta, t, p)>0$ some universal positive constant independent of $r$ and $k$. 

Since $0\leq\rho<r$, we have
$$
f_{\rho}=\frac{y'}{r},\quad f_{i}=0, \quad f_{\rho\rho}=\frac{y''}{r^2},\quad f_{\rho i}=0,\quad |f_{ij}|\leq Cr f_{\rho} \leq C y' \leq C.
$$
Thus, using that $\psi=f^{2/(n-2)}$ and $0<\delta\leq f\leq 1$, we get
\begin{equation}\label{e-dpsi}
C^{-1}\frac{y'}{r}\leq \psi_{\rho}\leq C\frac{y'}{r},\quad \psi_{i}=0, \quad |\psi_{\rho\rho}|\leq \frac{C}{r^2},\quad \psi_{\rho i}=0,\quad |\psi_{ij}|\leq Cr\psi_{\rho}\leq C y' \leq C.
\end{equation}
In particular
$$
C^{-1}\frac{(y')^2}{r^2}\leq|\nabla \psi|^2=\psi_{\rho}^2\leq C\frac{(y')^2}{r^2}.
$$

\medskip

\subsection*{Step 3.} From now on,  we consider indices $a,b=\rho,1,\ldots,n-1$, while $i,j=1,\ldots,n-1$. We will estimate the terms in \eqref{e-int} not involving the Weyl curvature, restricted to the ball $B_r$.

We have
$$
-\frac{R_{ab}\psi^a\psi^b}{\psi/k^2+|\nabla\psi|^2}=-\frac{R_{\rho\rho}\psi_{\rho}^2}{\psi/k^2+\psi_{\rho}^2}=-R_{\rho\rho}-\frac{1}{k^2}\frac{\psi R_{\rho\rho}}{\psi/k^2+\psi_{\rho}^2}\leq -R_{\rho\rho}+\frac{1}{k^2}\frac{C_1r^2}{r^2/k^2+C_2(y')^2}
$$
and thus
\begin{equation}\label{e-i1}
-\int_{B_r}\frac{R_{ab}\psi_a\psi_b}{\psi/k^2+|\nabla\psi|^2}f\,dV_g \leq C|B_r|+\frac{1}{k^2}\Theta
\end{equation}
where $|B_r|$ denotes the volume of $B_r$ and $\Theta=\Theta(p,1/k,r)>0$ will denote a continuous function in $1/k$ and $r$, for $0<r<r_0$ and $0\leq 1/k <1$. 

Also
$$
\frac{f_{ab}\psi^a\psi^b}{\psi/k^2+|\nabla\psi|^2}=\frac{f_{\rho\rho}\psi_{\rho}^2}{\psi/k^2+\psi_{\rho}^2}=f_{\rho\rho}-\frac{1}{k^2}\frac{\psi f_{\rho\rho}}{\psi/k^2+\psi_{\rho}^2}\leq \frac{y''}{r^2}+\frac{1}{k^2}\frac{C_1}{r^2/k^2+C_2(y')^2}
$$
and integrating over $B_r$, we get
\begin{equation}\label{e-i2}
\int_{B_r}\frac{f_{ab}\psi^a\psi^b}{\psi/k^2+|\nabla\psi|^2}\,dV_g \leq \frac{1}{r^2}\int_{B_r} y''\,dV_g+\frac{1}{k^2}\Theta.
\end{equation}

We have 
$$
\frac{f_a\psi^a}{\psi/k^2+|\nabla\psi|^2}\leq C \frac{\psi_{\rho}^2}{\psi/k^2+\psi_{\rho}^2}\leq C,\quad -\frac{\psi\Delta f}{\psi/k^2+|\nabla\psi|^2}\leq \frac{C_1}{r^2/k^2+C_2(y')^2}
$$
and therefore
\begin{equation}\label{e-i3}
\frac{1}{k^2}\frac{n-1}{2}\int_{B_r} \frac{f_a\psi^a}{\psi/k^2+|\nabla\psi|^2}\,dV_g-\frac{1}{k^2}\int_{B_r} \frac{\psi\Delta f}{\psi/k^2+|\nabla\psi|^2}\,dV_g \leq \frac{1}{k^2}\Theta.
\end{equation}

Moreover
\begin{align*}
\frac{\psi_{ab}\psi^b\psi_{ac}\psi^c}{(\psi/k^2+|\nabla\psi|^2)^2} -\frac{|\psi_{ab}\psi^a\psi^b|^2}{(\psi/k^2+|\nabla\psi|^2)^3} &= \frac{\psi_{\rho\rho}^2\psi_{\rho}^2}{(\psi/k^2+\psi_{\rho}^2)^2} -\frac{\psi_{\rho\rho}^2\psi_{\rho}^4}{(\psi/k^2+\psi_{\rho}^2)^3}=\frac{1}{k^2}\frac{\psi \psi_{\rho\rho}^2\psi_{\rho}^2}{(\psi/k^2+\psi_{\rho}^2)^3}\\
&\leq \frac{1}{k^2}\frac{C_1}{(r^2/k^2+C_2(y')^2)^3}
\end{align*}
and thus
\begin{equation}\label{e-i4}
\frac{n-1}{n-2}\int_{B_r} \left[\frac{\psi_{ab}\psi^b\psi_{ac}\psi^c}{(\psi/k^2+|\nabla\psi|^2)^2} -\frac{|\psi_{ab}\psi^a\psi^b|^2}{(\psi/k^2+|\nabla\psi|^2)^3} \right]f\,dV_{g}\leq  \frac{1}{k^2}\Theta.
\end{equation}

Finally, reasoning as before, one has 
\begin{equation}\label{e-i5}
\frac{1}{k^2}\frac{n-1}{n-2}\int_{B_r} \frac{\tfrac14 |\nabla\psi|^6-|\nabla\psi|^2(\psi_{ab}\psi^a\psi^b)\psi}{(\psi/k^2+|\nabla\psi|^2)^3}f\,dV_g \leq \frac{1}{k^2}\Theta.
\end{equation}
Therefore, since
$$
\int_{B_r} R_g f\,dV_g \leq C |B_r|,
$$
using \eqref{e-i1},\eqref{e-i2},\eqref{e-i3} and \eqref{e-i4} in \eqref{e-int}, we obtain that
\begin{align}\label{e-ie1}
\Phi_{B_r} \leq & \,t\int_{B_r}\frac{f}{\psi}\left( |W_{g}|_{\gb}+|E_{g}(2k\sqrt{\psi})|_{\gb}\right)\,dV_{g} +C|B_r|+\frac{1}{r^2}\int_{B_r} y''\,dV_g+\frac{1}{k^2}\Theta,
\end{align}
where $\Phi_{B_r}$ denotes the quantity defined in \eqref{e-int} restricted to $B_r$. Note that this intermediate estimate, when $t=0$, coincides with the one of Aubin in \cite{aub2}. 

\medskip

\subsection*{Step 4.} We now estimate the remaining terms in \eqref{e-int} which involve the Weyl curvature. Since
$$
\gb=g + d(2k\sqrt{\psi})\otimes d(2k\sqrt{\psi}),
$$
from Lemma \ref{l-a0}, we have
$$
\gb^{\rho\rho}=\frac{1}{1+4k^2 (\sqrt{\psi})_{\rho}^2},\quad \gb^{\rho i} =0, \quad \gb^{ij}=g^{ij}.
$$
Therefore, for any Riemann-type $4$-tensor, $T$, we obtain
\begin{equation}\label{e-ngb}
|T_{g}|^2_{\gb}=\sum_{i,j,k,t=1}^{n-1}T_{ijkt}^2+\frac{4}{1+4k^2 (\sqrt{\psi})_{\rho}^2}\sum_{i,k,t=1}^{n-1}T_{i\rho kt}^2 + \frac{4}{\left[1+4k^2 (\sqrt{\psi})_{\rho}^2\right]^2}\sum_{i,k=1}^{n-1}T_{i\rho k\rho}^2.
\end{equation}
In particular (this follows immediately from $\gb\geq g$):
$$
|W_g|_{\gb} \leq |W_g|_g\quad\text{and}\quad t\int_{B_r}\frac{f}{\psi}|W_g|_{\gb}\,dV_g \leq C |B_r|.
$$
From \eqref{e-ie1}, we obtain
\begin{equation}\label{e-ie2}
\Phi_{B_r} \leq  t\int_{B_r}\frac{f}{\psi}|E_{g}(2k\sqrt{\psi})|_{\gb}\,dV_{g} +C|B_r|+\frac{1}{r^2}\int_{B_r} y''\,dV_g+\frac{1}{k^2}\Theta.
\end{equation}
Concerning the first integral, we have the following key estimate:
\begin{lemma}\label{l-kest} We have
$$
t\int_{B_r}\frac{f}{\psi}|E_{g}(2k\sqrt{\psi})|_{\gb}\,dV_{g} \leq C|B_r|+\frac{1}{k^2}\Theta,
$$
for some $C=C(n,\delta, t, p)>0$ and $\Theta=\Theta(p,1/k,r)>0$ as above.
\end{lemma}
\begin{proof} We set $\eta=2\sqrt{\psi}$ and $E=E_g(2k\sqrt{\psi})=E_g(k\eta)$. From \eqref{e-dpsi}, since $0<\delta^{2/(n-2)}\leq \psi\leq 1$, we have
\begin{equation}\label{e-deta}
C^{-1}\frac{y'}{r}\leq \eta_{\rho}\leq C\frac{y'}{r},\quad \eta_{i}=0, \quad |\eta_{\rho\rho}|\leq \frac{C}{r^2},\quad \eta_{\rho i}=0,\quad |\eta_{ij}|\leq Cr\eta_{\rho}\leq C y' \leq C.
\end{equation}
Firstly, from Lemma \ref{l-a0} and \eqref{e-deta}, we get
\begin{align*}
E_{ijkt}&=\frac{k^2}{1+k^2 \eta_\rho^2}\pa{\eta_{ik}\eta_{jt}-\eta_{it}\eta_{jk}} \\    
&+\frac{k^2\eta_\rho^2}{(1+k^2 \eta_\rho^2)(n-2)}\pa{R_{i\rho k\rho}g_{jt}-R_{i\rho t\rho}g_{jk}+R_{j\rho t\rho}g_{ik}-R_{j\rho k\rho}g_{it}} \\  
&-\frac{2k^2R_{\rho\rho}\eta_\rho^2}{(1+k^2 \eta_\rho^2)(n-1)(n-2)}\pa{g_{ik}g_{jt}-g_{it}g_{jk}} \\  
&-\frac{k^2}{(1+k^2 \eta_\rho^2)(n-2)}\Big[\pa{(\Delta \eta) \eta_{ik}-\eta_{ip}\eta^p_k}g_{jt}-\pa{(\Delta \eta) \eta_{it}-\eta_{ip}\eta^p_t}g_{jk}\Big.\\
&\hspace{4cm}\Big.+\pa{(\Delta \eta) \eta_{jt}-\eta_{jp}\eta^p_t}g_{ik}-\pa{(\Delta \eta) \eta_{jk}-\eta_{jp}\eta^p_k}g_{it}\Big] \\  
&+\frac{k^2}{(1+k^2 \eta_\rho^2)(n-1)(n-2)}\sq{\pa{\Delta \eta}^2-\abs{\nabla^2 \eta}^2}\pa{g_{ik}g_{jt}-g_{it}g_{jk}} \\  
&+\frac{k^4\eta_\rho^2\eta_{\rho\rho}}{(1+k^2 \eta_\rho^2)^2(n-2)}\pa{\eta_{ik}g_{jt}-\eta_{it}g_{jk}+\eta_{jt}g_{ik}-\eta_{jk}g_{it}}  \\  
&-\frac{2k^4\eta_{\rho}^2\eta_{\rho\rho}}{(1+k^2 \eta_\rho^2)^2(n-1)(n-2)}\pa{\Delta \eta-\eta_{\rho\rho}}\pa{g_{ik}g_{jt}-g_{it}g_{jk}}.
\end{align*}
Since $\Delta\eta = \eta_{\rho\rho}+\eta_{p}^p$, we can simplify the expression, obtaining
\begin{align*}
E_{ijkt}&=\frac{k^2}{1+k^2 \eta_\rho^2}\pa{\eta_{ik}\eta_{jt}-\eta_{it}\eta_{jk}} \\    
&+\frac{k^2\eta_\rho^2}{(1+k^2 \eta_\rho^2)(n-2)}\pa{R_{i\rho k\rho}g_{jt}-R_{i\rho t\rho}g_{jk}+R_{j\rho t\rho}g_{ik}-R_{j\rho k\rho}g_{it}} \\  
&-\frac{2k^2R_{\rho\rho}\eta_\rho^2}{(1+k^2 \eta_\rho^2)(n-1)(n-2)}\pa{g_{ik}g_{jt}-g_{it}g_{jk}} \\  
&-\frac{k^2}{(1+k^2 \eta_\rho^2)(n-2)}\Big[\pa{\eta_{p}^p\eta_{ik}-\eta_{ip}\eta^p_k}g_{jt}-\pa{\eta_{p}^p \eta_{it}-\eta_{ip}\eta^p_t}g_{jk}\Big.\\
&\hspace{4cm}\Big.+\pa{\eta_{p}^p \eta_{jt}-\eta_{jp}\eta^p_t}g_{ik}-\pa{\eta_{p}^p \eta_{jk}-\eta_{jp}\eta^p_k}g_{it}\Big] \\  
&+\frac{k^2}{(1+k^2 \eta_\rho^2)(n-1)(n-2)}\sq{(\eta_{p}^p)^2+2\eta_{\rho\rho}\eta_{p}^p-|\eta_{ij}|^2}\pa{g_{ik}g_{jt}-g_{it}g_{jk}} \\  
&-\frac{k^2\eta_{\rho\rho}}{(1+k^2 \eta_\rho^2)^2(n-2)}\pa{\eta_{ik}g_{jt}-\eta_{it}g_{jk}+\eta_{jt}g_{ik}-\eta_{jk}g_{it}}  \\  
&-\frac{2k^4\eta_{\rho}^2\eta_{\rho\rho}\eta_{p}^p}{(1+k^2 \eta_\rho^2)^2(n-1)(n-2)}\pa{g_{ik}g_{jt}-g_{it}g_{jk}}.
\end{align*}
In particular, we have simplified the fourth block with the sixth one. Coupling the fifth block with the last one, we obtain
\begin{align*}
E_{ijkt}&=\frac{1}{1/k^2+\eta_\rho^2}\pa{\eta_{ik}\eta_{jt}-\eta_{it}\eta_{jk}} \\    
&+\frac{\eta_\rho^2}{(1/k^2+\eta_\rho^2)(n-2)}\pa{R_{i\rho k\rho}g_{jt}-R_{i\rho t\rho}g_{jk}+R_{j\rho t\rho}g_{ik}-R_{j\rho k\rho}g_{it}} \\  
&-\frac{2R_{\rho\rho}\eta_\rho^2}{(1/k^2+\eta_\rho^2)(n-1)(n-2)}\pa{g_{ik}g_{jt}-g_{it}g_{jk}} \\  
&-\frac{1}{(1/k^2+\eta_\rho^2)(n-2)}\Big[\pa{\eta_{p}^p\eta_{ik}-\eta_{ip}\eta^p_k}g_{jt}-\pa{\eta_{p}^p \eta_{it}-\eta_{ip}\eta^p_t}g_{jk}\Big.\\
&\hspace{4cm}\Big.+\pa{\eta_{p}^p \eta_{jt}-\eta_{jp}\eta^p_t}g_{ik}-\pa{\eta_{p}^p \eta_{jk}-\eta_{jp}\eta^p_k}g_{it}\Big] \\  
&+\frac{1}{(1/k^2+\eta_\rho^2)(n-1)(n-2)}\sq{(\eta_{p}^p)^2-|\eta_{ij}|^2}\pa{g_{ik}g_{jt}-g_{it}g_{jk}} \\  
&-\frac{1}{k^2}\frac{\eta_{\rho\rho}}{(1/k^2+\eta_\rho^2)^2(n-2)}\pa{\eta_{ik}g_{jt}-\eta_{it}g_{jk}+\eta_{jt}g_{ik}-\eta_{jk}g_{it}}  \\  
&+\frac{1}{k^2}\frac{2\eta_{\rho}^2\eta_{\rho\rho}\eta_{p}^p}{(1/k^2+\eta_\rho^2)^2(n-1)(n-2)}\pa{g_{ik}g_{jt}-g_{it}g_{jk}}.
\end{align*}
Using \eqref{e-deta}, since $|\eta_{ik}\eta_{jt}|\leq C \eta_{\rho}^2$, it is easy to see that the first five blocks are bounded by $C=C(n,\delta, t, p)>0$ while the last two are controlled by
$$
\frac{1}{k^2}\frac{C_1}{[r^2/k^2+C_2(y')^2]^2}.
$$
Therefore
\begin{equation}\label{e-eE1}
|E_{ijkt}| \leq C + \frac{1}{k^2}\frac{C_1}{[r^2/k^2+C_2(y')^2]^2}.
\end{equation}
Secondly, from Lemma \ref{l-a0} and \eqref{e-deta}, we get
\begin{equation}\label{e-eE2}
E_{i\rho kt} =0.
\end{equation}
Lastly, using again Lemma \ref{l-a0} and \eqref{e-deta}, we obtain
\begin{align*}
E_{i\rho k\rho}&=\frac{k^2\eta_{ik}\eta_{\rho\rho} }{1+k^2 \eta_\rho^2}+\frac{k^2R_{ik}\eta_{\rho}^2}{n-2}+\frac{k^2 R g_{ik}\eta_{\rho}^2}{(n-1)(n-2)}+\frac{k^2R_{i\rho k\rho}\eta_{\rho}^2}{n-2}-\frac{2k^2R_{\rho\rho}g_{ik}\eta_{\rho}^2}{(n-1)(n-2)}\\  
&-\frac{k^2}{n-2}\sq{(\Delta \eta) \eta_{ik}-\eta_{ip}\eta^p_k}-\frac{k^2g_{ik}\eta_{\rho\rho}}{(1+k^2 \eta_\rho^2)(n-2)}\pa{\Delta \eta-\eta_{\rho\rho}} \\
&+\frac{k^2g_{ik}}{(n-1)(n-2)}\sq{\pa{\Delta \eta}^2-\abs{\nabla^2 \eta}^2}\\  
&+\frac{k^4\eta_{\rho}^2\eta_{ik}\eta_{\rho\rho}}{(1+k^2 \eta_\rho^2)(n-2)}-\frac{2k^4g_{ik}\eta_{\rho}^2\eta_{\rho\rho}}{(1+k^2 \eta_\rho^2)(n-1)(n-2)}\pa{\Delta \eta-\eta_{\rho\rho}}.
\end{align*}
Since $\Delta\eta = \eta_{\rho\rho}+\eta_{p}^p$, we can simplify this expression, obtaining
\begin{align*}
E_{i\rho k\rho}&=\frac{k^2\eta_{ik}\eta_{\rho\rho} }{1+k^2 \eta_\rho^2}+\frac{k^2R_{ik}\eta_{\rho}^2}{n-2}+\frac{k^2 R g_{ik}\eta_{\rho}^2}{(n-1)(n-2)}+\frac{k^2R_{i\rho k\rho}\eta_{\rho}^2}{n-2}-\frac{2k^2R_{\rho\rho}g_{ik}\eta_{\rho}^2}{(n-1)(n-2)}\\  
&-\frac{k^2}{n-2}\sq{\eta_{\rho\rho}\eta_{ik}+\eta_p^p\eta_{ik}-\eta_{ip}\eta^p_k}-\frac{k^2g_{ik}\eta_{\rho\rho}\eta_p^p}{(1+k^2 \eta_\rho^2)(n-2)}\\
&+\frac{k^2g_{ik}}{(n-1)(n-2)}\sq{(\eta_{p}^p)^2+2\eta_{\rho\rho}\eta_{p}^p-|\eta_{ij}|^2}\\  
&+\frac{k^4\eta_{\rho}^2\eta_{ik}\eta_{\rho\rho}}{(1+k^2 \eta_\rho^2)(n-2)}-\frac{2k^4g_{ik}\eta_{\rho}^2\eta_{\rho\rho}\eta_p^p}{(1+k^2 \eta_\rho^2)(n-1)(n-2)}\\
&=\frac{k^2\eta_{ik}\eta_{\rho\rho} }{1+k^2 \eta_\rho^2}+\frac{k^2R_{ik}\eta_{\rho}^2}{n-2}+\frac{k^2 R g_{ik}\eta_{\rho}^2}{(n-1)(n-2)}+\frac{k^2R_{i\rho k\rho}\eta_{\rho}^2}{n-2}-\frac{2k^2R_{\rho\rho}g_{ik}\eta_{\rho}^2}{(n-1)(n-2)}\\  
&-\frac{k^2}{n-2}\sq{\eta_p^p\eta_{ik}-\eta_{ip}\eta^p_k}-\frac{k^2g_{ik}\eta_{\rho\rho}\eta_p^p}{(1+k^2 \eta_\rho^2)(n-2)}+\frac{k^2g_{ik}}{(n-1)(n-2)}\sq{(\eta_{p}^p)^2-|\eta_{ij}|^2}\\  
&+\frac{k^2\eta_{ik}\eta_{\rho\rho}}{(1+k^2 \eta_\rho^2)(n-2)}+\frac{k^2g_{ik}\eta_{\rho\rho}\eta_p^p}{(1+k^2 \eta_\rho^2)(n-1)(n-2)}.
\end{align*}
Rearranging the terms, we get
\begin{align*}
E_{i\rho k\rho}&=\frac{k^2R_{ik}\eta_{\rho}^2}{n-2}+\frac{k^2 R g_{ik}\eta_{\rho}^2}{(n-1)(n-2)}+\frac{k^2R_{i\rho k\rho}\eta_{\rho}^2}{n-2}-\frac{2k^2R_{\rho\rho}g_{ik}\eta_{\rho}^2}{(n-1)(n-2)}\\  
&-\frac{k^2}{n-2}\sq{\eta_p^p\eta_{ik}-\eta_{ip}\eta^p_k}+\frac{k^2g_{ik}}{(n-1)(n-2)}\sq{(\eta_{p}^p)^2-|\eta_{ij}|^2}\\
&+\frac{n-1}{n-2}\frac{k^2\eta_{ik}\eta_{\rho\rho} }{1+k^2 \eta_\rho^2}-\frac{k^2g_{ik}\eta_{\rho\rho}\eta_p^p}{(1+k^2 \eta_\rho^2)(n-1)}.
\end{align*}
Therefore, from \eqref{e-deta}, we deduce
$$
|E_{i\rho k\rho}| \leq C k^2 \eta_{\rho}^2+ \frac{C_1}{r^2/k^2+C_2(y')^2},
$$
and thus
\begin{equation}\label{e-eE3}
\frac{1}{1+k^2\eta_{\rho}^2}|E_{i\rho k\rho}| \leq C + \frac{1}{k^2}\frac{C_1}{\left[r^2/k^2+C_2(y')^2\right]^2}.
\end{equation}
As a consequence, using \eqref{e-ngb} and \eqref{e-eE1}, \eqref{e-eE2}, \eqref{e-eE3}, we obtain
$$
|E_{g}(2k\sqrt{\psi})|_{\gb} \leq C + \frac{1}{k^2}\frac{C_1}{\left[r^2/k^2+C_2(y')^2\right]^2}
$$
which implies
$$
t\int_{B_r}\frac{f}{\psi}|E_{g}(2k\sqrt{\psi})|_{\gb}\,dV_{g} \leq C|B_r|+\frac{1}{k^2}\Theta,
$$
for some $C=C(n,\delta, t, p)>0$ and $\Theta=\Theta(p,1/k,r)>0$.
\end{proof}

\medskip

\subsection*{Step 5.} Using Lemma \ref{l-kest} in \eqref{e-ie2}, we obtain 
\begin{equation}\label{e-ie3}
\Phi_{B_r} \leq  \,C|B_r|+\frac{1}{r^2}\int_{B_r} y''\,dV_g+\frac{1}{k^2}\Theta
\end{equation}
for some $C=C(n,\delta, t, p)>0$ and $\Theta=\Theta(p,1/k,r)>0$. Since, $y'(1)=0$, integrating by parts, we obtain
\begin{align*}
\frac{1}{r^2}\int_{B_r} y''\,dV_g&=-\frac{1}{r}\int_{B_r}y' \partial_{\rho}\log\sqrt{\text{det}g_{ij}}\,dV_g-\frac{n-1}{r}\int_{B_r}\frac{y'}{\rho}\,dV_g\\
&\leq \frac{C}{r}|B_r| -\frac{n-1}{r}\int_{B_r}\frac{y'}{\rho}\,dV_g.
\end{align*}
Hence, from \eqref{e-ie3}, we get
$$
\Phi_{B_r} \leq  \,C\left(1+\frac{1}{r}\right)|B_r|-\frac{n-1}{r}\int_{B_r}\frac{y'}{\rho}\,dV_g+\frac{1}{k^2}\Theta.
$$
Using that, by assumption, $y'(x)\geq 1$ for all $(1/4)^{1/(n-1)}\leq x\leq (3/4)^{1/(n-1)}$, we obtain
\begin{align*}
\Phi_{B_r}  &\leq  \,C\left(1+\frac{1}{r}\right)|B_r|-\frac{n-1}{r}|\SS^{n-1}|\inf_M \sqrt{\text{det}g_{ij}} \int_{r(\tfrac{1}{4})^{1/(n-1)}}^{r(\tfrac{3}{4})^{1/(n-1)}}\rho^{n-2}\,d\rho+\frac{1}{k^2}\Theta\\
&\leq C\left(1+\frac{1}{r}\right)|B_r|-\frac{C_2}{r^2}|B_r|+\frac{1}{k^2}\Theta,
\end{align*}
where we used the fact that $|B_r|\sim c r^n$ as $r\to 0$. In particular, there exist a continuous function $\lambda(p)> 0$ and, for $p\in M$ fixed, a continuous function $\Theta_p(r)>0$ in $r$, for $0<r<r_0$, such that
$$
\Theta(p, 1/k, t)\leq \Theta_p(r),
$$
and
\begin{equation}\label{e-ie4}
\Phi_{B_r}  \leq \left[C\left(1+\frac{1}{r}\right)-\frac{\lambda}{r^2}\right]|B_r| +\frac{1}{k^2}\Theta_p(r).
\end{equation}
Since, by assumption, $F_g=R_g+t|W_g|_g\geq 0$, given $\nu>0$, there exists a positive radius $0<r_1<r_0$ such that
\begin{equation}\label{e-r1}
\frac{\lambda}{r_1^2}-C\left(1+\frac{1}{r_1}\right)-1\geq\nu \overline{F}_g,
\end{equation}
where $\overline{F}_g:=\left(\int_M F_g\,dV_g\right)/\text{Vol}_g(M)$. Consider $h$ disjoint geodesic balls of radius $r_1$ centered at $p_j\in M$, $B^j_{r_1}(p_j)$, $j=1,\ldots,h$, such that
$$
\sum_{j=1}^h |B^j_{r_1}(p_j)| > \frac{1}{\nu} \text{Vol}_g(M).
$$
This is possible, for $\nu$ sufficiently large. On every ball $B^j$, we choose
$$
k^2:=\sup_{j=1,\ldots,h} \frac{\Theta_{p_j}(r_1)}{|B^j_{r_1}(p_j)| },
$$
if this supremum is larger than $1$, otherwise, we choose $k^2=1$. From \eqref{e-ie4} and \eqref{e-r1}, for all $j=1,\ldots,h$, we get
$$
\Phi_{B^j_{r_1}}\leq -\nu \overline{F}_g |B^j_{r_1}(p_j)| -|B^j_{r_1}(p_j)| +\frac{1}{k^2}\Theta_{p_j}(r_1) \leq -\nu \overline{F}_g |B^j_{r_1}(p_j)|.
$$
Therefore, since $f\equiv \psi\equiv 1$ on $M\setminus B^j$, for all $j=1,\ldots,h$, we obtain
$$
\Phi_M\leq \int_M F_g\,dV_g -\nu \overline{F}_g \sum_{j=1}^{h}|B^j_{r_1}(p_j)| <\overline{F}_g\left(\text{Vol}_g(M)-\nu\sum_{j=1}^{h}|B^j_{r_1}(p_j)|\right)\leq 0.
$$
This concludes the proof of Theorem \ref{t-main}. To be precise, we note that the proof above gives a $C^{2,\alpha}$ metric with negative constant scalar-Weyl curvature $F$. The density of smooth metrics in the space of $C^{2,\alpha}$ metrics (with the $C^{2,\alpha}$ norm) will then give us a smooth metric with negative scalar-Weyl curvature. From Lemma \ref{l-g2} we obtain a smooth metric with constant negative scalar-Weyl curvature.
\begin{flushright}$\square$\end{flushright}

\

\

\begin{ackn}
\noindent The author is member of the "Gruppo Nazionale per le Strutture Algebriche, Geometriche e le loro Applicazioni" (GNSAGA) of the Istituto Nazionale di Alta Matematica (INdAM).
\end{ackn}

\

\

\bibliographystyle{amsplain}

\

\

\parindent=0pt

\end{document}